\documentclass[12pt]{amsart}

\usepackage{amsmath,amsthm,amsfonts,amscd,amssymb}
\usepackage{times,euler,eucal,eufrak,graphicx}
\usepackage{pst-node}
\usepackage{pst-plot}

\usepackage{hyperref}
\usepackage{comment}

\title[A non-convex asymptotic quantum Horn body]{A non-convex asymptotic quantum Horn body}
\author[Collins]{Beno\^\i{}t Collins$^{\dagger}$}
\address{
Department of Mathematics and Statistics, University of Ottawa,
585 King Edward,
Ottawa, ON
K1N 6N5 Canada, and
CNRS, Department of Mathematics, Lyon 1 Claude Bernard University} 
\email{bcollins@uottawa.ca}
\thanks{\footnotesize $^{\dagger}$Research supported in part by NSERC
Discovery grant RGPIN/341303-2007 and the ANR GranMa and Galoisint grants.}

\author[Dykema]{Kenneth J.\ Dykema$^{*}$}
\address{K.\ Dykema, Department of Mathematics, Texas A\&M University,
College Station, TX 77843-3368, USA}
\email{kdykema@math.tamu.edu}
\thanks{\footnotesize $^{*}$Research supported in part by NSF grant DMS-0600814}

\theoremstyle{plain}
\newtheorem{lemma}{Lemma}[section]
\newtheorem{theorem}[lemma]{Theorem}
\newtheorem{proposition}[lemma]{Proposition}

\theoremstyle{definition}

\theoremstyle{remark}

\newcommand\Mcal{{\mathcal{M}}} 

\begin{document}

\begin{abstract}
We prove by a counterexample that asymptotic quantum Horn bodies are not convex in general.
\end{abstract}

\maketitle

\section{Introduction}

It is  known  that, given $A$ (resp. $B$) selfadjoint matrices 
in $\mathbb{M}_n(\mathbb{C})$ with
eigenvalues
$\lambda_1\geq\ldots \geq \lambda_n$ (resp. $\mu_1\geq\ldots \geq \mu_n$), 
the set of possible eigenvalues of $A+B$, denoted by
$\nu_1\geq \ldots \geq \nu_n$, is a convex polyhedron
of $\{(x_1\geq\ldots \geq x_n)\}\subset \mathbb{R}^n$.
This follows from results by Kirwan, Guillemin and Sternberg (see \cite{kirwan} and references therein).
The actual description of the polyhedron, conjectured by Horn in \cite{H62}
was proved to be true by several authors including Klyachko, Knutson and Tao
(see \cite{F00} and references therein).

The same question can be addressed in the case of a $II_1$ factor, 
namely, given $\lambda,\mu$ (compactly supported) real probability measures, 
what are the  probability measures $\nu$ such that there exists a $II_1$ factor
$M$ with selfadjoint elements $a$ (resp. $b$) in it of distribution
$\lambda$ (resp. $\mu$) such that $a+b$ has
distribution $\nu$.
This situation was studied at length under the additional assumption that
$M$ embeds in $R^{\omega}$ by Bercovici and Li in \cite{BL06}.
Recently it was proved in \cite{BCDLT} that
the assumption $M$ embeds in $R^{\omega}$ is actually not
needed. 
The paper \cite{CD} addressed a similar question where,
instead of considering $A+B$, one considers $a_1\otimes A+a_2\otimes B$ with $a_1$ and $a_2$ 
prescribed selfadjoint matrices.
One observes that this set is not convex in the sense above 
(example 4.3 in \cite{CD}).
This set is called `quantum Horn body' and
it was proved that this set scales asymptotically.
It was also proved in \cite{CD} that
all of these sets being asymptotically 
approximable by their finite dimensional versions
is equivalent to the Connes embedding problem. 
Note that this result, is not only  a reformulation of
the Connes embedding problem: it is rather an embeddability test
for a given $II_1$ factor.

However the geometry of this `quantum Horn body' was quite mysterious and
beyond closedness, nothing was known. 
We asked (Question 4.4, p. 638 of \cite{CD}) whether
the asmyptotic quantum Horn bodies $K^{a_1,a_2}_{\alpha,\beta,\infty}$ are always convex.
The aim of the present paper is to describe in detail one class of examples,
showing that they are not convex in general.

The paper is organized as follows: 
in Section \ref{sec:notations}, we first recall a few notations.
In Section \ref{sec:counterexample},
we exhibit and study our counterexample.
Finally, we end with a few comments and additional remarks.

\section{Notations and known facts}
\label{sec:notations}

Let $\mathbb{R}^N_{\ge}$ denote the set of $N$--tuples
of real numbers listed in nonincreasing order.
The {\em eigenvalue sequence} of an $N\times N$ self--adjoint matrix is its sequence of eigenvalues
repeated according to multiplicity and in nonincreasing order, so as to lie in $\mathbb{R}^N_\ge$.
Consider $\alpha=(\alpha_1,\ldots,\alpha_N)$ and $\beta=(\beta_1,\ldots,\beta_N)$
in $\mathbb{R}^N_{\ge}$.
Let $S_{\alpha,\beta}$ be the set of all possible
eigenvalue sequences $\gamma=(\gamma_1,\ldots,\gamma_N)$ of $A+B$, where $A$ and $B$
are self--adjoint $N\times N$ matrices with eigenvalue sequences $\alpha$ and $\beta$, respectively.
Klyatchko, Totaro, Knutson and Tao described the set $S_{\alpha,\beta}$
in terms first conjectured by Horn.
See Fulton's exposition~\cite{F00}.
We call $S_{\alpha,\beta}$  the {\em Horn body} of $\alpha$ and $\beta$;
It is a closed, convex subset of $\mathbb{R}^N_\ge$.

Let $\mathcal{F}$ be the set of all right--continuous, nonincreasing, bounded functions
$\lambda:[0,1)\to\mathbb{R}$.
Let $\mathcal{M}$ be a von Neumann algebra with normal, faithful, tracial state $\tau$ and let $a=a^*\in\mathcal{M}$.
The {\em distribution} of $a$ is the Borel measure $\mu_a$, supported on the spectrum of $a$,
such that 
\begin{equation}
\tau(a^n)=\int_{\mathbb{R}}t^n\,d\mu_a(t)\qquad(n\ge1).
\end{equation}
The {\em eigenvalue function} of $a$ is 
$\lambda_a\in\mathcal{F}$ defined by
\begin{equation}
\lambda_a(t)=\sup\{x\in\mathbb{R}\mid\mu_a((x,\infty))>t\}.
\end{equation}
Thus, $\mu_a$ is the Lebesgue--Stieltjes measure arising from the unique nondecreasing,
right--continuous function $H:\mathbb{R}\to[0,1]$ such that $H(\lambda(t))=1-t$ at points $t$ where
$\lambda$ is continuous.
Moreover, if $g:\mathbb{R}\to\mathbb{C}$ is continuous, then
\[
\int g\,d\mu_a=\int_0^1g(\lambda_a(t))\,dt.
\]
We call $\mathcal{F}$ the set of all eigenvalue functions.
It is an affine space, where we take scalar multiples and sums of functions in the usual way.

Let
$\Mcal_1^+(\mathbb{R})_c$ denote the set of all compactly supported Borel probability measures on the real line
and let $\mathrm{EV}:\Mcal_1^+(\mathbb{R})_c\to\mathcal{F}$ be the identification given by $\mu_a\mapsto\lambda_a$,
as described above.
Since $\Mcal_1^+(\mathbb{R})_c$ is a subspace of the dual of the algebra $C(\mathbb{R})$ of all continuous functions on $\mathbb{R}$,
we endow $\mathcal{F}$ with the weak$^*$--topology inherited from this pairing.

Let $N\in\mathbb{N}$ and $\alpha,\beta\in\mathbb{R}^N_\ge$.
For $d\in\mathbb{N}$, let
\begin{equation}
K_{\alpha,\beta,d}=\{\lambda_C\mid C=\mathrm{diag}(\alpha)\otimes 1_d+U(\mathrm{diag}(\beta)\otimes1_d)U^*,\,U\in\mathbb{U}_{Nd}\},
\end{equation}
and
\begin{equation}
K_{\alpha,\beta,\infty}=\overline{\bigcup_{d\ge1}K_{\alpha,\beta,d}}\,.
\end{equation}
where the closure is taken according to the weak$^*$--topology on $\mathcal{F}$.
This set was considered by Bercovici and Li~\cite{BL01}, \cite{BL06}
as an infinite dimensional limit of the sets $S_{\alpha,\beta}$.

Let $a_1,a_2\in\mathbb{M}_n(\mathbb{C})_{s.a.}$, and $\alpha,\beta\in\mathbb{R}^N_\ge$.
We consider the set $K^{a_1,a_2}_{\alpha,\beta}$ of the eigenvalue functions of all
matrices of the form
\begin{equation}\label{eq:aUDU}
a_1\otimes\mathrm{diag}(\alpha)+a_2\otimes U\mathrm{diag}(\beta)U^*,\qquad(U\in\mathbb{U}_N).
\end{equation}
We view $K^{a_1,a_2}_{\alpha,\beta}$ as a subset of $\mathcal{F}$
and we may equally well consider the corresponding eigenvalue sequences
and view $K^{a_1,a_2}_{\alpha,\beta}$ as a subset of $\mathbb{R}^{nN}_\ge$.
The set $K^{a_1,a_2}_{\alpha,\beta}$ is seen to be
the analogue of the Horn body $S_{\alpha,\beta}$, but with ``coefficients'' $a_1$ and $a_2$.
We will refer to these sets as {\em quantum Horn bodies}.

Extending the notions introduced above,
for integers $d\ge1$, 
let $K^{a_1,a_2}_{\alpha,\beta,d}$ be the set of
the eigenvalue functions of all matrices of the form
\begin{equation}\label{eq:AUDdU}
a_1\otimes\mathrm{diag}(\alpha)\otimes 1_d+a_2\otimes U(\mathrm{diag}(\beta)\otimes 1_d)U^*,
\qquad(U\in\mathbb{U}_{Nd}).
\end{equation}
If $d'$ divides $d$, then we have
\begin{equation}
K^{a_1,a_2}_{\alpha,\beta,d'}\subseteq K^{a_1,a_2}_{\alpha,\beta,d}\;.
\end{equation}
Let us define
\begin{equation}
K^{a_1,a_2}_{\alpha,\beta,\infty}=\overline{\bigcup_{d\in\mathbb{N}}K^{a_1,a_2}_{\alpha,\beta,d}}\;,
\end{equation}
where the closure is in the weak$^*$--topology for $\mathcal{F}$ described earlier in this section.
Note that the set $K^{a_1,a_2}_{\alpha,\beta,\infty}$ is compact.
We call it \emph{asymptotic quantum Horn body}.

We know from \cite{CD}, Example 4.3,
that $K^{a_1,a_2}_{\alpha,\beta}$ need not be convex,
and we asked whether
it is true that $K^{a_1,a_2}_{\alpha,\beta,\infty}$ must be convex,
or even that $K^{a_1,a_2}_{\alpha,\beta,d}$ must be convex for all $d$ sufficiently large.
(We recall that the convexity we are considering here is with respect to the
affine structure of pointwise addition and scalar mulitplication of real--valued
functions on $[0,1]$.
This is not the same as the affine structure obtained by identifying elements of $\mathcal{F}$
with probability measures on $\mathbb{R}$ and performing vector space operations on measures.)

\section{The counterexample}
\label{sec:counterexample}

We show that $K^{a_1,a_2}_{\alpha,\beta,\infty}$ is not convex when
$\alpha=\beta = (1,0)\in \mathbb{R}^2_{\geq}$ and
the coefficients are 
$$a_1=\begin{pmatrix}1 & 0 \\ 0 & -1\end{pmatrix},
\quad
a_2=\begin{pmatrix}2s-1 & 2\sqrt{s(1-s)}\\ 2\sqrt{s(1-s)} & 1-2s
\end{pmatrix}.$$
Note that both these coefficient matrices are selfadjoint and unitary. 
Their eigenvalues are $\{1,-1\}$ so they are conjugate to 
each other. 
The parameter $s$ takes values in $[0,1]$ and 
these matrices don't commute unless $s\in \{0,1\}$.

If $p$ and $q$ are projections in some $M_{2d}(\mathbb{C})$,
each of normalized trace $1/2$, then $\mathbb{C}^{2d}$
can be written as a direct sum of $d$ subspaces, each of dimension $2$
and each reducing for both $p$ and $q$.
Thus, $p$ and $q$ can be taken to be block diagonal, with $2\times2$ blocks $p_i$ and $q_i$, respectively.
Furthermore, after a change of basis, each of these blocks can be taken of the form
\[
q_i=\begin{pmatrix}1 & 0 \\ 0 & 0\end{pmatrix}
\qquad
p_i=\begin{pmatrix}t_i & \sqrt{t_i(1-t_i)}\\ \sqrt{t_i(1-t_i)} & 1-t_i \end{pmatrix},
\]
for $0\le t_i\le 1$.

Let us consider one such block, and let us write $t$ for $t_i$.
We have
\begin{align*}
a_1\otimes p_i&=\begin{pmatrix}1 & 0 \\ 0 & -1\end{pmatrix}\otimes
\begin{pmatrix}t & \sqrt{t(1-t)}\\ \sqrt{t(1-t)} & 1-t \end{pmatrix} \\[2ex]
a_2\otimes q_i&=
\begin{pmatrix}2s-1 & 2\sqrt{s(1-s)}\\ 2\sqrt{s(1-s)} & 1-2s \end{pmatrix}\otimes
\begin{pmatrix}1 & 0\\ 0 & 0 \end{pmatrix} 
\end{align*}
and
$$
a_1\otimes p_i + a_2\otimes q_i=
\begin{pmatrix}
-1+2s+t & 2\sqrt{(1-s)s} & \sqrt{(1-t)t} & 0\\
 2\sqrt{(1-s)s}  & 1-2s-t & 0 & -\sqrt{(1-t)t}\\
 \sqrt{(1-t)t} & 0 & 1-t & 0 \\
 0 & -\sqrt{(1-t)t} & 0 & -1+t
\end{pmatrix}.
$$

A direct computation shows that
the characteristic polynomial of this matrix is
$$P(\lambda )=(1-t)^2-2(1-t+2st)\lambda^2+\lambda^4.$$
This fourth degree equation has only terms of even degree
and can be solved as a compound second degree equation.
The eigenvalues of $a_1\otimes p_i + a_2\otimes q_i$, in decreasing order, are
as follows:
\begin{align*}
\lambda_1&=\sqrt{1-t+2st+2\sqrt{st-st^2+s^2t^2}} \\
\lambda_2&=\sqrt{1-t+2st-2\sqrt{st-st^2+s^2t^2}} \\
\lambda_3&=-\sqrt{1-t+2st-2\sqrt{st-st^2+s^2t^2}} \\
\lambda_4&=-\sqrt{1-t+2st+2\sqrt{st-st^2+s^2t^2}}
\end{align*}
We regard and $\lambda_1,\ldots,\lambda_4$ as
a functions of $s$ and $t$.
Regarding $s$ as fixed, let
$$\nu_t=\frac{1}{4}\sum_{i=1}^4 \delta_{\lambda_i (s,t)}.$$
Let $\Phi_s:\Mcal_1^+([0,1])\to\Mcal_1^+(\mathbb{R})_c$
be the affine and continuous extension of the map $\delta_t\mapsto\nu_t$.
The above discussion implies:

\begin{proposition}
The asymptotic quantum Horn body $K^{a_1,a_2}_{\alpha,\beta,\infty}$ is
the image of $\mathcal{M}_1^+([0,1])$ under the map
$\mathrm{EV}\circ\Phi_s$.
\end{proposition}

Our main result is:

\begin{theorem}
For any choice of $s\in (0,1)$,
the asymptotic quantum Horn body $K^{a_1,a_2}_{\alpha,\beta,\infty}$ is not convex.
\end{theorem}
\begin{proof}
We have
\begin{align*}
\Phi_s(\delta_0)=\nu_0&=\frac12\delta_1+\frac12\delta_{-1} \\
\Phi_s(\delta_1)=\nu_1&=\frac14\delta_{2\sqrt s}+\frac12\delta_0+\frac14\delta_{-2\sqrt s}.
\end{align*}
We will show that some convex combination
\[
r\mathrm{EV}\circ\Phi_s(\delta_0)+(1-r)\mathrm{EV}\circ\Phi_s(\delta_1),
\]
$0<r<1$, does not lie in the image of $\mathrm{EV}\circ\Phi_s$.
The eigenvalue functions in question are constant on the intervals
$[0,\frac14)$, $[\frac14,\frac12)$, $[\frac12,\frac34)$ and $[\frac34,1)$,
and their values there are indicated in Table~\ref{tab:EVs}.
\begin{table}[h]
\caption{Values of the eigenvalue functions on intervals}
\label{tab:EVs}
\begin{tabular}{r||c|c|c|c}
\rule{0ex}{2.5ex}
 &$[0,\frac14)$&$[\frac14,\frac12)$&$[\frac12,\frac34)$&$[\frac34,1)$ \\ \hline
\rule{0ex}{2.5ex}
 $\mathrm{EV}\circ\Phi_s(\delta_0)$
 &$2\sqrt s$&$0$&$0$&$-2\sqrt s$ \\ \hline
\rule{0ex}{2.5ex}
 $\mathrm{EV}\circ\Phi_s(\delta_1)$
 &$1$&$1$&$-1$&$-1$ \\ \hline
\rule{0ex}{2.5ex}
 $\begin{aligned}&r\mathrm{EV}\circ\Phi_s(\delta_0) \\ &+(1-r)\mathrm{EV}\circ\Phi_s(\delta_1)\end{aligned}$
 &$(1-r)+2r\sqrt s$&$1-r$&$r-1$&$r-1-2r\sqrt s$ \\ \hline
\end{tabular}
\end{table}

We have
\[
r\mathrm{EV}\circ\Phi_s(\delta_0)+(1-r)\mathrm{EV}\circ\Phi_s(\delta_1)
=\mathrm{EV}\big(\frac14(\delta_{1-r+2r\sqrt s}+\delta_{1-r}+\delta_{r-1}+\delta_{r-1-2r\sqrt s})\big)
\]
and it will suffice to show that for some $r\in(0,1)$, the measure
\[
\sigma=\frac14(\delta_{1-r+2r\sqrt s}+\delta_{1-r}+\delta_{r-1}+\delta_{r-1-2r\sqrt s})
\]
is not in the image of $\Phi_s$.
For this, it will suffice to show that for some $r\in(0,1)$ and for all $t\in[0,1]$, we have
$\mathrm{supp}(\nu_t)\not\subseteq\mathrm{supp}(\sigma)$.

If
$\mathrm{supp}(\nu_t)\subseteq\mathrm{supp}(\sigma)$,
then we have either (a) $t=0$ and either $r=0$ or $s=1/4$ or
(b) the following equations hold:
\begin{align}
1-r+2r\sqrt s&=\sqrt{1-t+2st+2\sqrt{st-st^2+s^2t^2}} \label{eq:nutsig1} \\
1-r&=\sqrt{1-t+2st-2\sqrt{st-st^2+s^2t^2}}. \label{eq:nutsig2}
\end{align}
Assume for the moment $s\ne1/4$.
Then $\mathrm{supp}(\nu_t)\subseteq\mathrm{supp}(\sigma)$ implies
that equations~\eqref{eq:nutsig1}--\eqref{eq:nutsig2} hold, and this implies that the polynomials
\begin{align*}
p_1&=\begin{aligned}[t]
      &r^4-4 r^3-4 r^2 s t+2 r^2 t+4 r^2+8 r s t-4 r t+4 s^2 t^2 \\
      &-8 s^2 t+4 s^2-4 s t^2+4 s t-4 s+t^2
     \end{aligned}
    \\
p_2&=\begin{aligned}[t]
      &r^4-4 r^3-2 r^2 s t-2 r^2 s+2 r^2 t^2-2 r^2 t+6 r^2+4 r s t \\
      &+4 r s-4 r t^2+4 r t-4 r+s^2 t^2-2 s^2 t+s^2-2 s t^3+4 s t^2 \\
      &-4 s t-2 s+t^4-2 t^3+3 t^2-2 t+1
     \end{aligned}
\end{align*}
both vanish.
However, a Gr\"obner basis for the ideal $I$ generated by $p_1$ and $p_2$ in $\mathbb{C}[r,s,t]$,
computed with respect to an elimination order, reveals that $I\cap\mathbb{C}[r,s]$ is the ideal generated
by the polynomial
\begin{equation}\label{eq:prs}
\begin{aligned}
&\big(r-1\big)^2\,
 \big(r^2-2 r-4 s+1\big)\,
 \big(r^4-4 r^3+4 r^2 s^2-6 r^2 s \\
&\;+6r^2-8 r s^2+12 r s-4 r+4 s^4-4 s^3+5 s^2-6 s+1\big)\, \\
&\; \big(r^6-6r^5+4 r^4 s^2-10 r^4 s +15 r^4-16 r^3 s^2+40 r^3 s \\
&\;-20 r^3+4 r^2 s^4+108r^2 s^3-79 r^2 s^2-28 r^2 s+15 r^2-8 r s^4 \\
&\; -216 r s^3+190 r s^2-24 r s
  -6 r-144 s^5+340 s^4-184 s^3 \\
&\;+13 s^2+6 s+1\big),
\end{aligned}
\end{equation}
where the factors are irreducible.
This implies that,  for every value of $s$, except possibly $s=1/4$,
choosing $r\in(0,1)$ so that the above polynomial~\eqref{eq:prs} does not vanish,
we have $\mathrm{supp}(\nu_t)\not\subseteq\mathrm{supp}(\sigma)$
for every $t\in[0,1]$.

Now supposing $s=1/4$, if $r\in(0,1)$ is such that the polynomial~\eqref{eq:prs} does not vanish,
then there is exactly one value of $t$ such that
$\mathrm{supp}(\nu_t)\subseteq\mathrm{supp}(\sigma)$ holds, namely $t=0$.
However, since $\sigma$ is not itself equal to $\nu_0$, it does not lie in the image of $\Phi_{1/4}$.
\end{proof}

\section{Discussion and concluding remarks}

The result above relies on the fact that the description of the representations 
of the $*$-algebra generated by two representations are particularly easy to 
understand. 
It is easy to generalize the above counterexample by modifying the
values of $a_1,a_2$ although formal computations become
more involved. 
It would be interesting to find a necessary and sufficient
criterion on $a_1,a_2$ in this case for the quantum Horn body
to be convex or not.

However, it is difficult to generalize the above counterexample 
to other sorts of $\lambda$ and $\mu$. Indeed, 
we do not know how to classify the representations of the $*$-algebra
generated by two elements such that at least one of them has a spectrum of
strictly more than two points.

We still wonder whether there exists `purely' asymptotic quantum Horn bodies that
are convex. 

\bibliographystyle{plain}

\begin{thebibliography}{9}

\bibitem{BL01} {\scshape H. Bercovici and W.S. Li,}
Inequalities for eigenvalues of sums in a von Neumann algebra,
\emph{Recent advances in operator theory and related topics (Szeged, 1999)},
Oper. Theory Adv.\ Appl.\ {\bf 127} 
Birkh\"auser, Basel, 2001, pp.\ 113--126.

\bibitem{BL06} {\scshape H. Bercovici and W.S. Li,}
Eigenvalue inequalities in an embeddable factor,
\emph{Proc.\ Amer.\ Math.\ Soc.}\ {\bf 134} (2006), 75--80.

\bibitem{BCDLT} {\scshape H. Bercovici, B. Collins, K. Dykema, W.S. Li and D. Timotin,}
Intersections of Schubert varieties and eigenvalue inequalities in an arbitrary finite factor
\emph{arXiv:0805.4817}\ 


\bibitem{CD} {\scshape B. Collins, K. Dykema,}
A linearization of Connes' embedding problem,
{\em New York J.\ Math.}\ {\bf 14} (2008), 617--641.


\bibitem{F00} {\scshape W. Fulton,}
Eigenvalues, invariant factors, highest weights, and
Schubert calculus,
{\em Bull. Amer. Math. Soc. (N.S.)}  {\bf 37} (2000), 209--249.

\bibitem{H62}
{\scshape A. Horn}
Eigenvalues of sums of Hermitian matrices,
{\em Pacific J. Math.} {\bf 12} (1962), 225-241.

\bibitem{kirwan}
{\scshape F. Kirwan}
Convexity properties of the moment mapping, III 
{\em Invent. math.} {\bf 77} (1984), 547-552.

\end{thebibliography}

\end{document}